\documentclass[10pt]{article}
\font\smallit=cmti10
\font\smalltt=cmtt10

\usepackage{amssymb,latexsym,amsmath,epsfig,amsthm}
\usepackage{amsmath, amscd, amssymb, amsthm, xspace, graphicx}
\usepackage[all]{xy}

\makeatletter

\renewcommand\section{\@startsection {section}{1}{\z@}
	{-30pt \@plus -1ex \@minus -.2ex}
	{2.3ex \@plus.2ex}
	{\normalfont\normalsize\bfseries}}

\renewcommand\subsection{\@startsection{subsection}{2}{\z@}
	{-3.25ex\@plus -1ex \@minus -.2ex}
	{1.5ex \@plus .2ex}
	{\normalfont\normalsize\bfseries}}

\renewcommand{\@seccntformat}[1]{\csname the#1\endcsname. }

\makeatother

\newtheorem{theorem}{Theorem}[section]

\newtheorem{conjecture}{Conjecture}
\newtheorem{proposition}{Proposition}

\theoremstyle{definition}
\newtheorem{definition}{Definition}

\begin{document}

\begin{center}
	\uppercase{\bf Consecutive Integers and the Collatz Conjecture}
	\vskip 20pt
	{\bf Marcus Elia}\\
	{\smallit Department of Mathematics, SUNY Geneseo, Geneseo, NY}\\
	{\tt mse1@geneseo.edu}\\ 
	\vskip 10pt
	{\bf Amanda Tucker}\\
	{\smallit Department of Mathematics, University of Rochester, Rochester, NY}\\
	{\tt amanda.tucker@rochester.edu}\\ 
\end{center}
\vskip 30pt
\vskip 30pt

\centerline{\bf Abstract}
\noindent
 Pairs of consecutive integers have the same height in the Collatz problem with surprising frequency. Garner gave a conjectural family of conditions for exactly when this occurs. Our main result is an infinite family of counterexamples to Garner's conjecture. 
\pagestyle{myheadings} 
\markright{\smalltt INTEGERS: 15 (2015)\hfill} 
\thispagestyle{empty} 
\baselineskip=12.875pt 
\vskip 30pt

\section{Introduction} \label{intro}

The Collatz function $C$ is a recursively defined function on the positive integers given by the following definition.
\[
C^k(n)=
\begin{cases}
n, \text{ if } k=0\\
C^{k-1}(n)/2, \text{ if }  C^{k-1}(n)\text{ is even}\\
3*C^{k-1}(n)+1, \text{ if } C^{k-1}(n)\text{ is odd}.
\end{cases}
\]
The famed Collatz conjecture states that, under the Collatz map, every positive integer converges to one~\cite{Lagarias}.
The \emph{trajectory} of a number is the path it takes to reach one.  For example, the trajectory of three is
\begin{equation*}
3\rightarrow 10\rightarrow 5\rightarrow 16\rightarrow 8\rightarrow 4\rightarrow 2\rightarrow 1.
\end{equation*}
The \emph{parity vector} of a number is its trajectory considered modulo two. So the parity vector of three is
\begin{equation*}
\langle 1,0,1,0,0,0,0,1\rangle.
\end{equation*}

Because applying the map  $n \mapsto 3n+1$ to an odd number will always yield an even number, it is sometimes more convenient to use the following alternate definition of the Collatz map, often called $T$ in the literature.
\[
T^k(n)=
\begin{cases}
n, \text{ if } k=0\\
T^{k-1}(n)/2, \text{ if }  T^{k-1}(n)\text{ is even}\\
(3*T^{k-1}(n)+1)/2, \text{ if } T^{k-1}(n)\text{ is odd}.
\end{cases}
\]
With this new definition, the trajectory of three becomes
\begin{equation*}
3\rightarrow 5\rightarrow 8\rightarrow 4\rightarrow 2\rightarrow 1
\end{equation*}
and its $T$ parity vector is 
\begin{equation*}
\langle 1,1,0,0,0,1\rangle.
\end{equation*}

Since the Collatz conjecture states that, for every positive integer $n$, there exists a non-negative integer $k$ such that $C^k(n)=1$, it is natural to ask for the smallest such value of $k$. This $k$ is called the \emph{height} of $n$ and denoted H$(n)$. So, for example, the height of three is seven because it requires seven iterations of the map $C$ for three to reach seven. In this paper, height is used only in association with the map $C$, never the map $T$.

It turns out that consecutive integers frequently have the same height.  Garner made a conjecture that attempts to predict, in terms of the map $T$ and its parity vectors, exactly which pairs have the same height~\cite{Garner}. He proved that his condition is sufficient to guarantee two consecutive numbers will have the same height, but only surmised that it is a necessary condition.

The main idea in this paper is that phrasing Garner's conjecture in terms of the map $C$ reveals an easier-to-verify implication of Garner's conjecture, namely, that if two consecutive integers have the same height, then they must reach $4$ and $5\pmod 8$ at the same step of their trajectory (see Proposition~\ref{GarnerEquivalent}).  Because this condition is much easier to check than the conclusion of Garner's conjecture, we were able to find an infinite family of pairs of consecutive integers that do not satisfy this condition, and, hence, constitute counterexamples to Garner's conjecture (see Theorem~\ref{MainTheorem}).

\begin{bf}Acknowledgements:\end{bf}
This research was made possible by an Undergraduate  Research Fellowship  from the Research Foundation for SUNY. In addition, we would like to thank Jeff Lagarias and Steven J. Miller for helpful conversations.  
We would also like to thank the referee for careful reading and helpful suggestions and corrections.

\section{Heights of consecutive integers}\label{heights}
Recall that the smallest non-negative $k$  such that $C^k(n)=1$ is called the \emph{height} of $n$ and denoted H$(n)$. The following is a graph of the height $H$ as a function of $n$.
\begin{center}
\includegraphics[height=2.4in]{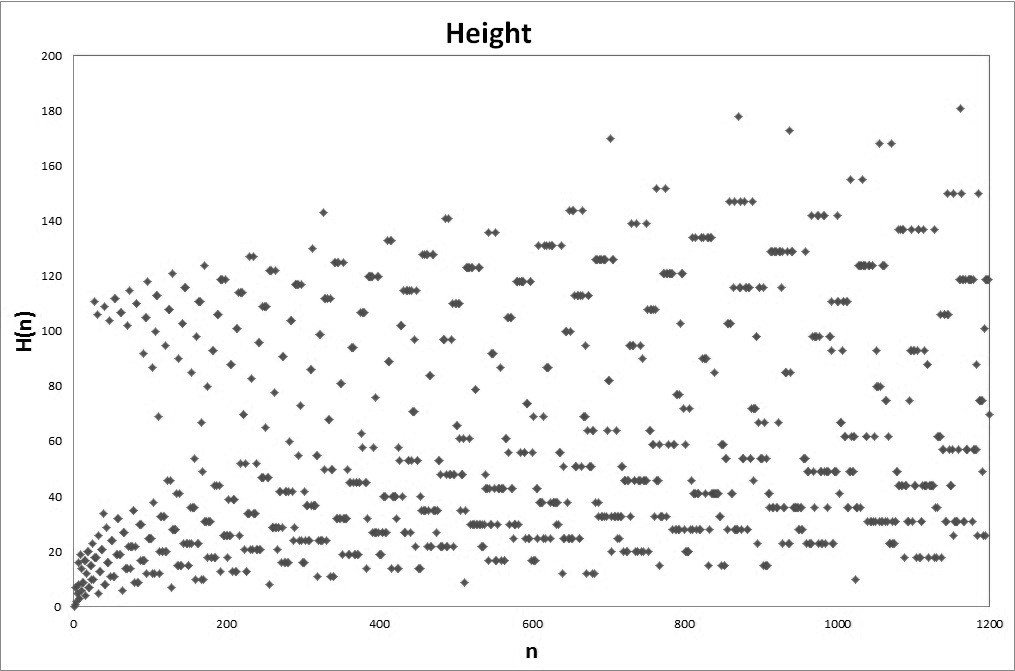}
\end{center}

The striking regularity in the above graph is the starting point for our studies, but remains largely elusive. If one na\"ively searches for curves of best fit to the visible curves therein, one quickly runs into a problem.  What appear to be distinct points  in the above graph are actually clusters of points, as can be seen below. Thus, it is not entirely clear which points one ought to work with when trying to find a curve of best fit.

\begin{center}
 \includegraphics[height=2.4in]{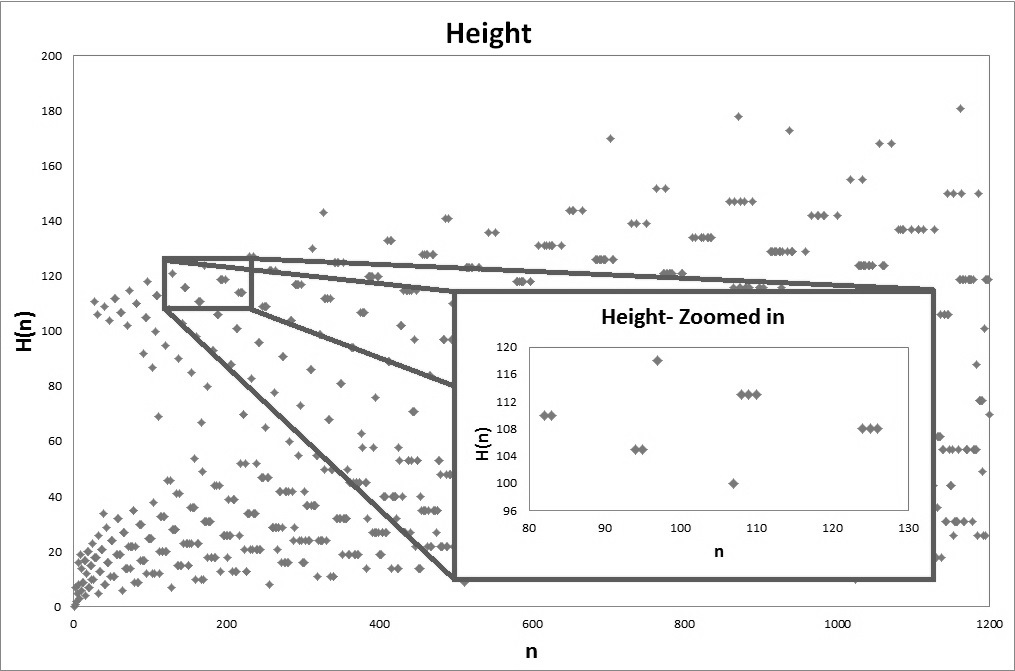}\\
\end{center}

 This leads to the surprising observation that many consecutive integers have the same height.  This is counterintuitive because if two integers are consecutive then they are of opposite parity, so the Collatz map initially causes one to increase ($n \mapsto 3n+1$) and the other to decrease ($n \rightarrow \frac{n}{2}$).  How, then, do they reach one in the same number of iterations? We give a sufficient congruence condition to guarantee two consecutive numbers will have the same height, and show that an all-encompassing theorem like Garner conjectured in~\cite{Garner} is not possible. In fact, we show the situation is much more complicated than Garner originally thought.
 
 The first pair of consecutive integers with the same height is twelve  and thirteen.  We see that for both numbers, $C^3(n)=10$.  Clearly, once their trajectories coincide, they will stay together and have the same height. This happens because twelve follows the path\\
 \begin{equation*}
 12\rightarrow 6\rightarrow 3 \rightarrow 10,
 \end{equation*}
 and thirteen follows the path\\
 \begin{equation*}
 13\rightarrow 40\rightarrow 20\rightarrow 10.
 \end{equation*}
 Now we seek to generalize this.  It turns out that twelve and thirteen merely form  the first example of a general phenomenon, namely, numbers that are $4$  and $5\pmod{8}$  always coincide after the third iteration.  The following result agrees with what Garner found using parity vectors~\cite{Garner}.

 \begin{theorem}
 If $n>4$ is congruent to $4\pmod{8}$, then $n$ and $n+1$ coincide at the third iteration and, hence, have the same height.
 \end{theorem}
 \begin{proof}
 Suppose $n>4$ and  $n\equiv 4\pmod{8}$. Then $n=8k+4$, for some $k\in\mathbb{N}$.  Then, because $8k+4$ and $4k+2$ are even, while $2k+1$ is odd,  the trajectory of $n$ under the map $C$ is
 \[
 8k + 4 \rightarrow 4k+2 \rightarrow 2k+1 \rightarrow 6k+4.
 \]
 Because $n+1=8k+5$ is odd, and $24k+16$ and $12k+8$ are even, the trajectory of $n+1$ under the map $C$ is
 
 \[
 8k+5 \rightarrow 24k+16 \rightarrow 12k+8  \rightarrow 6k+4. 
 \]
Therefore, $n$ and $n+1$ coincide at the third iteration.
 \end{proof}

  \section{Garner's conjecture}
 Garner wanted to generalize this to predict all possible pairs of consecutive integers that coincide.   Since he used the map $T$ (defined in Section~\ref{intro}) instead of the map $C$, we will do the same in this section except during the proof of Proposition~\ref{GarnerEquivalent}. He observed that whenever two consecutive integers have the same height, their parity vectors appear to end in certain pairs of corresponding \emph{stems} immediately before coinciding.  He defined a \emph{stem} as a parity vector of the form
 \begin{equation*}
 s_i = \langle 0,\underbrace{ 1, 1, ..., 1}_{i\:\: 1's}, 0, 1\rangle,
 \end{equation*}
 and the \emph{corresponding stem} as
 \begin{equation*}
 s_i'=\langle 1,\underbrace{1,1,...,1}_{i\: \: 1's},0,0\rangle.
 \end{equation*}
  LaTourette used the following definitions of a stem and a block in her senior thesis~\cite{LaTourette}, which we adhere to here as well. In what follows, we write $T_w(n)$ to mean apply the sequence of steps indicated by the parity vector $w$ to the input $n$ using the map $T$.
  \begin{definition}
 (LaTourette) A pair of parity sequences $s$ and $s'$ of length $k$ are \emph{corresponding stems} if, for any integer $x$, $T_s(x) = T_{s'}(x+1)$ and, for any initial subsequences $v$ and $v'$ of $s$ and $s'$ of equal length, $|T_v(x)-T_{v'}(x+1)|\ne 1$ and $T_v(x)\ne T_{v'}(x+1)$.
  \end{definition}
 \begin{definition}
 (LaTourette) A \emph{block} prefix is a pair of parity sequences $b$ and $b'$, each of length $k$, such that for all positive integers $x$, $T_b(x)+1 = T_{b'}(x+1)$.
 \end{definition}
  In his conclusion, Garner conjectured that all  corresponding stems will be of the form $s_i$ and $s_i'$ listed above. LaTourette conjectured the same. 
  \begin{conjecture}~\label{GarnConj}
  (Garner) Any pair of consecutive integers of the same height will have parity vectors for the non-overlapping parts of their trajectories ending in $s_i$ and $s_i'$~\cite{Garner}. 
   \end{conjecture}  
   
   Garner gave no bound on the length of stem involved, though, so searching for counterexamples by computer was a lengthy task. The big innovation in this paper is that using the map $C$ instead of the map $T$ yields a much simple implication of Garner's conjecture, which makes it possible to search for counterexamples.   
   
   \begin{proposition}\label{GarnerEquivalent}
    If $n$ and $n+1$ have parity vectors for the non-overlapping parts of their trajectories ending in $s_i$ and $s_i'$, and $k$ is the smallest positive integer such that $C^k(n)=C^k(n+1)$, then $C^{k-3}(n)\equiv 4\pmod{8}$ and $C^{k-3}(n+1)=C^{k-3}(n)+1$ or  $C^{k-3}(n+1)\equiv 4\pmod{8}$ and $C^{k-3}(n)=C^{k-3}(n+1)+1$.
   \end{proposition}
   \begin{proof}
   To see this, we must change the Garner stems to be consistent with the map $C$.  Converting the parity vectors simply involves inserting an extra `0' after each `1'.  So Garner's stems in terms of the map $C$ now look like
    \begin{equation*}
    s_i = \langle 0,\underbrace{ 1,0, 1,0, ..., 1,0}_{i\:\: 1,0's}, 0, 1,0\rangle,
    \end{equation*}
    and
    \begin{equation*}
    s_i'=\langle 1,0,\underbrace{1,0,1,0,...,1,0}_{i\: \: 1,0's},0,0\rangle.
    \end{equation*}
    Now we will rearrange this more strategically. We have
    \begin{equation*}
    s_i = \langle \underbrace{0, 1,0, 1, ..., 0,1}_{i\:\: 0,1's},0, 0, 1,0\rangle,
    \end{equation*}
    and
    \begin{equation*}
    s_i'=\langle \underbrace{1,0,1,0,...,1,0}_{i\: \: 1,0's},1,0,0,0\rangle.
    \end{equation*}
    
    The point of these stems is that the trajectories coincide right after this vector.  Since both end with a `0', they have coincided one step before the end, so we can simply omit the last `0'.  Now the corresponding stems are only $\langle 0,0,1\rangle$ and $\langle 1,0,0\rangle$, with repeated blocks in front of them. Terras\cite{Terras} proved that there is a bijection between the set of integers modulo $2^k$ and the set of parity vectors of length $k$. The algorithm to get from a parity vector of length 3 to an integer modulo 8 is explicit, so we can easily determine that numbers with those parity vectors are congruent to 4 and 5$\pmod{8}$, respectively. 
    
    Let $j$ be the point at which they coincide, so $C^k(n) = C^k(n+1) = j$. Applying $C^{-1}$ to $j$ as prescribed by both  $\langle 0,0,1\rangle$ and $\langle 1,0,0\rangle$ yields $\frac{4j-1}{3}-1$ and $\frac{4j-1}{3}$, respectively. Thus, we see that $C^{k-3}(n+1)=C^{k-3}(n)+1$. An identical argument yields the case where $C^{k-3}(n+1)\equiv 4\pmod{8}$, and we get $C^{k-3}(n)=C^{k-3}(n+1)+1$ in that case as well.
    \end{proof}
 So, written in terms of the map $C$,  all of Garner's other stems are simply repeated blocks of `$01$' and `$10$' in front of the stems $\langle 0,0,1\rangle$ and $\langle 1,0,0\rangle$. This is the benefit of applying the map $C$ in this situation. It is now feasible to check if a pair of consecutive integers is a counterexample to Garner's conjecture.  Suppose $n$ and $n+1$ have the same height.  According to Garner's conjecture, $n$ and $n+1$ would have $T$ parity vectors before coinciding that end in $s_i$ and $s_i'$.  By Proposition \ref{GarnerEquivalent}, this would in turn imply that $n$ and $n+1$ have $C$ parity vectors ending in $\langle 0,0,1\rangle$ and $\langle 1,0,0\rangle$.  Therefore, if we find a pair of positive integers $n$ and $n+1$ such that their parity vectors do not end in $\langle 0,0,1\rangle$ and $\langle 1,0,0\rangle$, we have found a counterexample to Garner's conjecture.\\
 \section{A counterexample to Garner's conjecture}
 We initially believed Garner's conjecture, but have since found many counterexamples. The first counterexample is the pair 3067 and 3068. The $C$-parity vector of $3067$ before coinciding with $3068$  is
 \begin{eqnarray*}
 \langle 1, 0, 1, 0, 0, 1, 0, 1, 0, 0, 1, 0, 1, 0, 0, 1, 0, 0, 1, 0, 0, 0, 1, 
 0, 0, 0, 1 \rangle,
 \end{eqnarray*}
 and that of $3068$ is 
 \begin{eqnarray*}
 \langle 0, 0, 1, 0, 1, 0, 1, 0, 1, 0, 1, 0, 1, 0, 1, 0, 1, 0, 0, 1, 0, 0, 1, 
 0, 0, 0, 0\rangle.
 \end{eqnarray*}
 By inspection, the parity vectors do not end with  $\langle 0,0,1\rangle$ and $\langle 1,0,0\rangle$ as Garner predicted. Thus, Garner's conjecture is false.\\

  A computer search found that there are 946 counterexample pairs less than a million. For numbers less than 5 billion, $0.214$\% of pairs of consecutive integers of the same height are counterexamples. By a simple argument, we can see that there must be infinitely many counterexample pairs.  \\
  \begin{theorem}\label{MainTheorem}
  There are infinitely many counterexamples to Garner's conjecture.
  \end{theorem}
  \begin{proof}
  Consider the parity vectors of 3067 and 3068 up to the point where they coincide. We know that there will be a pair with the same parity vectors for every integer of the form $2^{19}m + 3067$ by Terras's bijection\cite{Terras}. Each of these pairs will coincide in the same way that 3067 and 3068 do and, thus, have the same height. Therefore, there are infinitely many counterexamples to Garner's conjecture.
  \end{proof}
  
   \section{Conclusion}
 At this point, we look at those numbers that do not have the stems Garner predicted to see why they coincide. To salvage Garner's conjecture, we seek to expand the list of possible stems.  To see what is going on, we have no choice but to examine the trajectories of 3067 and 3068, side by side (See Appendix A).

We can see that there are no other places within the trajectories where their values have a difference of one.  Therefore, by the current definition of a stem, the entire parity vector of length 27 (up until they coincide at 1384) is a new stem.  However, by this logic, the next counterexample, 4088 and 4089, has a new stem of length 30.  The next pair, 6135 and 6136, has a stem of length 28.  It would be ridiculous to have only one stem (of length 3) before 3067 and to suddenly add dozens more of varying lengths.  Instead, we look for some new type of stems within these counterexample, stems that do not start with consecutive integers.  The trajectories of all three pairs listed above coincide at 1384. In fact, they have the same 22 elements leading up to that.  Thus, it is tempting to label that beginning as the stem.  But if we look further, the consecutive integers 32743 and 32744 join that group just 5 steps before coinciding at 1384. Therefore, the situation is much more complicated than Garner's stems. It would be interesting to know if there is some pattern similar to what Garner conjectured, perhaps with a much-expanded list of stems, that explains every pair of consecutive numbers that converges together. However, we have found no such simple salvage of Garner's conjecture.

 We have shown that pairs of integers of the form $8m+4$ and $8m+5$ have coinciding trajectories after 3 steps (and therefore have the same height).  We have also shown that all pairs that obey Garner's conjecture ultimately reduce down to the $4$ and $5 \pmod{8}$ case before coinciding. This allowed us to find that 3067 and 3068 form the smallest of an infinite family of counterexamples to Garner's longstanding conjecture~\cite{Garner}.

 \section{Appendix A}
 This chart shows the partial trajectories of the first eight counterexample pairs to Garner's conjecture.\\
 \begin{center}
 \includegraphics[width=5.4in]{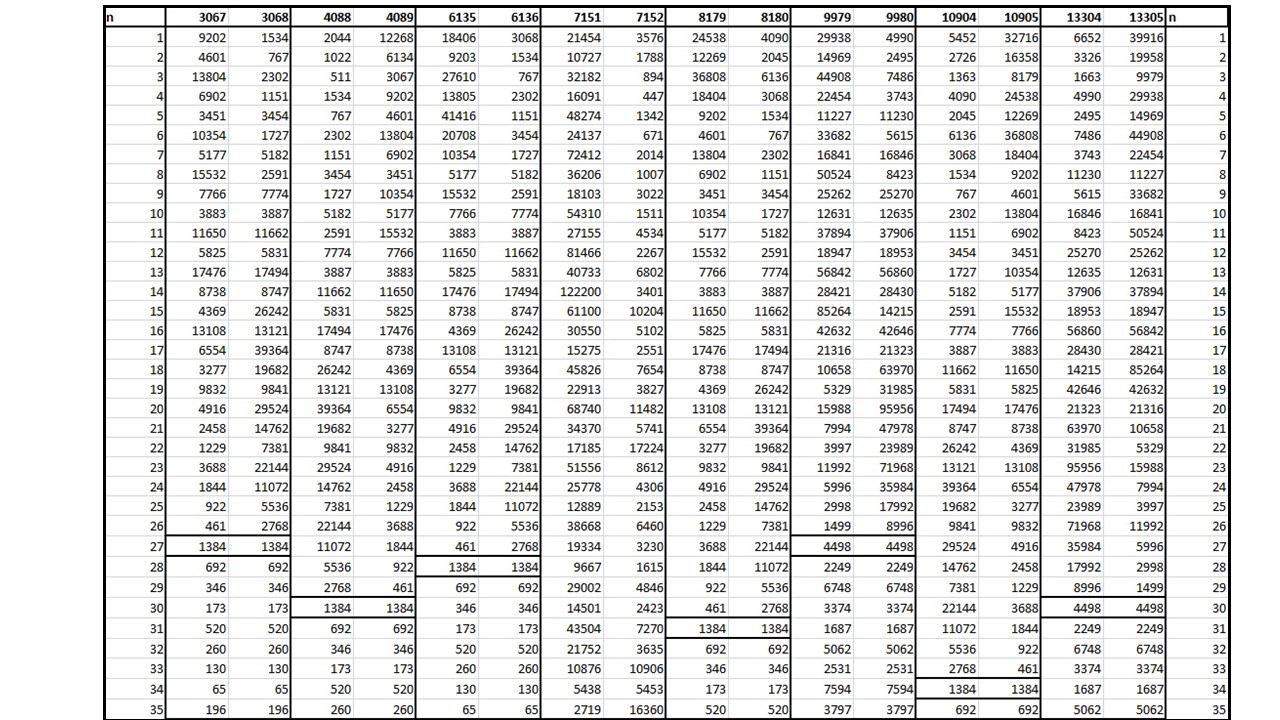}
\end{center}

\end{document}